\documentclass[12pt,a4paper,reqno]{amsart} 
\usepackage{amssymb}
\usepackage{latexsym}
\usepackage{amsmath}
\usepackage{mathrsfs}
\usepackage{color}
\usepackage{calc}                         
\usepackage{cancel}

\newcommand{\rr}[1]{\mathbf R^{#1}}
\newcommand{\nn}[1]{\mathbf N^{#1}}

\newcommand{\cc}[1]{\mathbf C^{#1}}

\newcommand{\nm}[2]{\Vert #1\Vert _{#2}}

\newcommand{\sets}[2]{\{ \, #1\, ;\, #2\, \} }
\newcommand{\ep}{\varepsilon}
\newcommand{\fy}{\varphi}

\newcommand{\norm}[1]{\Vert#1\Vert }

\newcommand{\eabs}[1]{\langle #1\rangle}

\newcommand{\vrum}{\vspace{0.1cm}}

\newcommand{\maclJ}{\mathcal J}
\newcommand{\maclA}{\mathcal A}
\newcommand{\maclH}{\mathcal H}
\newcommand{\maclS}{\mathcal S}
\newcommand{\mascS}{\mathscr S}
\newcommand{\bsySig}{\boldsymbol \Sigma}
\newcommand{\bsycalS}{\boldsymbol {\mathcal S}}

\newtheorem{thm}{Theorem}
\newtheorem{prop}[thm]{Proposition}

\newtheorem{lemma}[thm]{Lemma}

\theoremstyle{definition}

\theoremstyle{remark}

\newtheorem{rem}[thm]{Remark}              

\author{Carmen Fern\'andez}

\address{Departament d' An\`{a}lisi Matem\`{a}tica, Universitat de Val\`{e}ncia, Valencia, Spain}

\email{fernand@uv.es}

\author{Antonio Galbis}

\address{Departament d' An\`{a}lisi Matem\`{a}tica, Universitat de Val\`{e}ncia, Valencia, Spain}

\email{antonio.galbis@uv.es}

\author{Joachim Toft}

\address{Department of Mathematics,
Linn{\ae}us University, V{\"a}xj{\"o}, Sweden}

\email{joachim.toft@lnu.se}

\thanks{The research of C. Fern\'andez and A. Galbis was partially supported by the projects  MTM2013-43540-P and GVA Prometeo II/2013/013 (Spain).}

\title{The Bargmann transform and powers of harmonic oscillator on
Gelfand-Shilov subspaces}

\begin{document}

\begin{abstract}
We consider the counter images $\maclJ (\rr d)$ and $\maclJ _0(\rr d)$
of entire functions with exponential and almost
exponential bounds, respectively, under the Bargmann transform, and we characterize them by estimates of powers of the harmonic
oscillator. We also consider the Pilipovi{\'c} spaces $\bsycalS _s(\rr d)$
and $\bsySig _s(\rr d)$ when $0<s<1/2$ and  deduce their images under the Bargmann transform.
\end{abstract}

\maketitle

\section{Introduction}\label{sec0}

\par

The aim of the paper is to characterize the images of the 
Pilipovi{\'c} spaces $\bsySig _s(\rr d)$ and $\bsycalS _{\! s}(\rr d)$
under the Bargmann transform when $s< 1/2$, as well as the test
function spaces $\maclJ _0(\rr d)$
and $\maclJ (\rr d)$\footnote{In the latest version of \cite{Toft14}, the spaces
$\maclJ _0(\rr d)$ and $\maclJ (\rr d)$ are denoted by $\maclH _{0,\flat}(\rr d)$
and $\maclH _\flat (\rr d)$, respectively.}, considered in \cite{Toft14},
in terms of estimates of powers of the harmonic oscillator. The set
$\maclJ (\rr d)$ consists of all $f\in \mascS (\rr d)$ such that their
Hermite series expansions are given by
\begin{equation}\label{fHermite}
f=\sum _{\alpha \in \nn d}c_\alpha (f)h_\alpha ,
\end{equation}
where
\begin{equation}\label{JHermCoeffCond}
|c_\alpha (f)|\le \frac {Cr^{|\alpha |}}{\sqrt {\alpha !}},
\end{equation}
for some constants $r>0$ and $C>0$. (See Section \ref{sec1} for notations.)
In the same way, $f$ belongs to $\maclJ _0(\rr d)$, if and only
if for every $r>0$, there is a constant $C>0$ such that \eqref{JHermCoeffCond}
holds.

\par

The sets $\maclJ _0(\rr d)$ and $\maclJ (\rr d)$ are small in the sense that
they are continuously embedded in the Schwartz space $\mascS (\rr d)$
and its subspaces $\maclS _{s_1}(\rr d)$ and $\Sigma _{s_2}(\rr d)$ when
$s_1\ge 1/2$ and $s_2>1/2$. Here $\maclS _s(\rr d)$ and $\Sigma _s(\rr d)$
are the sets of Gelfand-Shilov spaces of Roumieu and Beurling types, respectively,
of order $s\ge 0$ on $\rr d$. 
The function spaces $\bsycalS _{\! s}(\rr d)$ $\bsySig _s(\rr d)$, 
increase
with the parameter $s\ge 0$, and
the following holds true (see \cite{Toft14} for
the verifications):
\begin{itemize}
\item $\bsycalS _{\! s}(\rr d)$ is non-trivial for every $s\ge 0$ and
$\bsycalS _0(\rr d)$
consists of all finite linear combinations of Hermite functions.
Furthermore, $\bsycalS _{\! s}(\rr d)=\maclS _s(\rr d)$ when $s\ge 1/2$, and
$\bsycalS _{\! s}(\rr d)\neq \maclS _s(\rr d) = \{ 0\}$ when $s< 1/2$;

\vrum

\item $\bsySig _s(\rr d)$ is non-trivial if and only if $s> 0$. Furthermore,
$\bsySig _s(\rr d)=\Sigma _s(\rr d)$ when $s> 1/2$, and
$\bsySig _s(\rr d)\neq \Sigma _s(\rr d) = \{ 0\}$ when $s\le 1/2$;

\vrum

\item For every
$$
\ep >0,\quad s\ge 0,\quad s_1<\frac 12 \quad \text{and}\quad
s_2\ge \frac 12,
$$
the inclusions
\begin{equation}\label{Eq:InclusionSpaces}
\begin{aligned}
\bsySig _s(\rr d)&\subseteq \bsycalS _{\! s}(\rr d)\subseteq \bsySig _{s+\ep}(\rr d)\subseteq \mascS (\rr d)
\\[1ex]
\text{and}\quad
\bsycalS _{\!  s_1}(\rr d)&\subseteq \maclJ _0(\rr d)\subseteq \maclJ (\rr d)
\subseteq \bsySig _{s_2}(\rr d)
\end{aligned}
\end{equation}
are continuous and dense. A similar fact holds for corresponding distribution spaces,
after the inclusions have been reversed.
\end{itemize}
\par

The spaces in \eqref{Eq:InclusionSpaces} and their duals have in most
of the cases,
convenient images under the Bargmann transform. In fact, in view of
\cite{Toft14} it is
proved that the Bargmann transform $\mathfrak V_d$ is injective on
all these spaces and their duals, and, among others, that 
\begin{align*}
\mathfrak V_d(\maclJ _0(\rr d)) &= \sets {F\in A(\cc d)}{|F(z)|\lesssim e^{R|z|}\
\text{for every}\ R>0},
\\[1ex]
\mathfrak V_d(\maclJ (\rr d)) &= \sets {F\in A(\cc d)}{|F(z)|\lesssim e^{R|z|}\
\text{for some}\ R>0},
\\[1ex]
\mathfrak V_d(\bsySig _{\frac 12}(\rr d)) &= \sets {F\in A(\cc d)}{|F(z)|\lesssim
e^{R|z|^2}, \ \text{for every}\ R>0},
\\[1ex]
\mathfrak V_d(\bsySig _{\frac 12}'(\rr d)) &= \sets {F\in A(\cc d)}{|F(z)|\lesssim
e^{R|z|^2}, \ \text{for some}\ R>0},
\\[1ex]
\mathfrak V_d(\maclJ '(\rr d)) &= A(\cc d),
\\[1ex]
\mathfrak V_d(\maclJ _0'(\rr d)) &= \bigcup _{r>0}A(B_r(0)).
\end{align*}
(See \cite{Toft14} for more comprehensive lists of mapping properties of the spaces in
\eqref{Eq:InclusionSpaces} and their duals under the Bargmann transform.) 
Here $A(\Omega )$ is the set of all analytic functions on the open set
$\Omega \subseteq \cc d$, and $B_r(z_0)$ is the open ball with center at
$z_0\in \cc d$ and radius $r>0$. We remark that one of the reasons for
considering $\maclJ _0(\rr d)$ and $\maclJ (\rr d)$ is that $\maclJ _0'(\rr d)$
and $\maclJ '(\rr d)$ possess the mapping properties
under the Bargmann transform given here above.

\par

In Section \ref{sec3} we make the list here above more complete by proving that
$$
\mathfrak V_d(\bsySig _{s}(\rr d)) = \sets {F\in A(\cc d)}{|F(z)|\lesssim
e^{R(\log \eabs z)^{\frac 1{1-2s}}}, \ \text{for every}\ R>0},
$$
and
$$
\mathfrak V_d(\bsycalS _{\! s}(\rr d)) = \sets {F\in A(\cc d)}{|F(z)|\lesssim
e^{R(\log \eabs z)^{\frac 1{1-2s}}}, \ \text{for some}\ R>0},
$$ when $0 < s < 1/2.$

\par

\section{Preliminaries}\label{sec1}

\par

In this section we recall some basic facts. We start by discussing
Pilipovi{\'c} and Gelfand-Shilov spaces and some of their properties. Finally
we recall the Bargmann transform and some of its
mapping properties.

\par

Let $0<h,s,t\in \mathbf R$. Then $\mathcal S_{s,h}(\rr d)$ consists of all
$f\in C^\infty (\rr d)$ such that
\begin{equation}\label{gfseminorm}
\nm f{\mathcal S_{s,h}}:= \sup \frac {|x^\beta \partial ^\alpha
f(x)|}{h^{|\alpha  + \beta |}(\alpha !\, \beta !)^s}
\end{equation}
is finite. The \emph{Gelfand-Shilov spaces} $\mathcal S_s(\rr d)$ 
and $\Sigma _s(\rr d)$, of Roumieu and Beurling 
types respectively, are the sets
\begin{equation}\label{GSspacecond1}
\mathcal S_s(\rr d) = \bigcup _{h>0}\mathcal S_{s,h}(\rr d)
\quad \text{and}\quad \Sigma _s(\rr d) =\bigcap _{h>0}
\mathcal S_{s,h}(\rr d),
\end{equation}
with inductive and projective topologies, respectively.
We remark that 
$\Sigma _s(\rr d)\neq \{ 0\}$, if and only if $s>1/2$, and
$\maclS _s(\rr d)\neq \{ 0\}$, if and only
if $s\ge 1/2$. We refer to \cite{ChuChuKim,GS} for
general facts about Gelfand-Shilov spaces, and their duals.

\par

Next we consider spaces which are obtained by suitable
estimates of Gelfand-Shilov or Gevrey type, after the operator
$x^\beta \partial ^\alpha $ in \eqref{gfseminorm} is replaced by
powers of the harmonic oscillator $H=|x|^2-\Delta$. More precisely,
if $s\ge 1/2$ ($s>1/2$), then Pilipovi{\'c} showed in \cite{Pil2}
that $f\in \maclS _s(\rr d)$ ($f\in \Sigma _s(\rr d)$), if and only if
\begin{equation}\label{GFHarmCond}
\sup _{N\ge 0} \frac {\nm{H^Nf}{L^\infty}}{h^N(N!)^{2s}}<\infty ,
\end{equation}
holds for some $h>0$ (for every $h>0$). (See also \cite{Pil,GrPiRo1} for more
general approaches.)
On the other hand, $\maclS _s(\rr d)$ ($\Sigma _s(\rr d)$) is empty when $s<1/2$
($s\le 1/2$), while any Hermite function $h_\alpha$ fulfills \eqref{GFHarmCond}
for some $h>0$ (for every $h>0$), when $s\ge 0$ ($s>0$).

\par

For this reason, we let $\bsycalS _{\! s}(\rr d)$ ($\bsySig _s(\rr d)$) be the set of all
$f\in C^\infty (\rr d)$ such that \eqref{GFHarmCond} holds for some $h>0$
(for every $h>0$). 
We call $\bsycalS _{\! s}(\rr d)$ and $\bsySig _s(\rr d)$ the \emph{Pilipovi{\'c}
spaces} of Roumieu and Beurling types respectively, of order $s\ge 0$ on $\rr d$.

\par

In \cite{GrPiRo1,Pil,Toft14} there are different types of characterizations of the
Pilipovi{\'c} spaces. For example, it is here proved that $f$ belongs to
$\bsycalS _{\! s}(\rr d)$ ($\bsySig _s(\rr d)$) with $s>0$, if and only if \eqref{fHermite} holds
with
$$
|c_\alpha (f)|\lesssim e^{-r|\alpha |^{1/2s}}
$$
for some $r>0$ (for every $r>0$).

\par

\section{Characterizations of $\maclJ _0(\rr d)$
and $\maclJ (\rr d)$ in terms of powers of the
harmonic oscillator}\label{sec2}

\par

In the previous section, $\bsycalS _{\! s}(\rr d)$ and $\bsySig _s(\rr d)$
were defined by means of \eqref{GFHarmCond}. In this section 
 we deduce characterizations of the test function spaces
$\maclJ _0(\rr d)$ and $\maclJ (\rr d)$, in similar ways.

\par

More precisely we have the following.

\par

\begin{thm}\label{thm:mainthm1}
Let $f$ be given by \eqref{fHermite}. Then the following conditions
are equivalent:
\begin{enumerate}
\item There exists $r > 0$ such that
$$
\left | c_\alpha (f)\right | \lesssim r^{|\alpha |}(\alpha !)^{-\frac{1}{2}}.
$$
\item There exists $r > 0$ such that
$$
\norm {H^N f}_{L^2} \lesssim 2^N r^{\frac{N}{\log N}}
\left ( \frac{2N}{\log N} \right )^{N ( 1-\frac{1}{\log N} ) }.
$$
\end{enumerate}
\end{thm}

\par

The same arguments also give the following.

\par 
 
\begin{thm}\label{thm:mainthm2}
Let $f$ be given by \eqref{fHermite}. Then the following conditions
are equivalent:
\begin{enumerate}
\item For every  $r > 0$,
$$
\left | c_\alpha (f) \right | \lesssim r^\alpha (\alpha !)^{-\frac{1}{2}}.
$$
\item For every  $r > 0$,
$$
\norm{H^N f} _{L^2} \lesssim 2^N r^{\frac{N}{\log N}}
\left ( \frac{2N}{\log N} \right )^{N ( 1-\frac{1}{\log N} ) }.
$$
\end{enumerate}
\end{thm}

\par

We need some preparations for the proofs, and start with the
following lemma.

\par

\begin{lemma}\label{lem:interval}
Let $r > 0$. Then, for $N$ large enough, the function
$$
f(t) = \frac {t^{2N}(2re)^t}{t^t},\ t > 0,
$$
attains its maximum in the interval $[t_1, t_2],$ where
\begin{equation}\label{talphaDef}
t_\alpha = \frac{2N}{\log N}\left(1 + \alpha \frac{\log (r\log N)}{\log N + 1}\right),
\quad \alpha \ge 0.
\end{equation}
\end{lemma}

\par

\begin{proof}
We may assume that $N>e^{1/r}$.
Let
$$
m(t) = \log f(t) = 2N\log t + t\log (2re) - t\log t.
$$
Then
$$
m'(t) = \frac{2N}{t} + \log(2r) - \log t
\quad \text{and}\quad
m''(t) = -\frac{2N}{t^2}- \frac{1}{t} < 0.
$$
It follows that $m$ is strictly concave, and has at most one local maximum,
and if so, it is also a global maximum. Furthermore, $m'(t)$ is strictly decreasing
and has a possible zero only in the possible point where maximum of $m(t)$
is attained. Hence it suffices to show that $m'(t_1) > 0$ and $m'(t_2) < 0$.

\par

We observe that $m'(t_0)>0$ and that the tangent line to the graph of $m'$
at point $(t_0, m'(t_0))$ intersects with the abscissa axis at $(t_1,0)$. Since
$m'$ is a convex function it follows that $m'(t_1) > 0$.

\par

In order to show that $m'(t_2) < 0$, we put
\begin{align*}
h(s) & = s+\left (  1+2\frac {\log (rs)}{s+1}  \right )(\log (rs) -s))
\\[1ex]
&= \log (rs) \left ( 1+\frac {2\log (rs)}{s+1} -\frac {2s}{s+1} \right ).
\end{align*}
Then $h(s)<0$ when $s=\log N$ is large enough, and by straight-forward
computations we get
\begin{align*}
m'(t_2) & = \left ( 1+2\frac {\log (rs)}{s+1} \right )^{-1}h(s)
-\log \left (1+2 \frac {\log (rs)}{s+1}\right )
\\[1ex]
&< \left ( 1+2\frac {\log (rs)}{s+1} \right )^{-1}h(s) <0,
\end{align*}
where the inequalities follow from the fact that $N>e^{1/r}$ is chosen large enough.
Hence $m'(t_2) < 0$.
\end{proof}

\par


\par

\begin{lemma}\label{lem:maximum}
Let $r> 0$ and let
$$
f(t) = \frac {t^{2N}(2re)^t}{t^t}.
$$
Then there exists a positive and increasing function $\theta$ on $[0,\infty )$
and an integer $N_0(r)$ such that
\begin{equation}\label{maxf(t)Est}
\begin{gathered}
\max_{t > 0} f(t)\leq \left(\frac{2N}{\log N}\right)
^{2N(1 - \frac{1}{\log N})} (\theta (r)\cdot r)^\frac{2N}{\log N},
\\[1ex]
\text{when}\quad N \geq N_0(r).
\end{gathered}
\end{equation}
\end{lemma}

\par

\begin{proof}
Let $N_0(r)>\max (e^{1/r},e)$ be such that the conclusions in Lemma
\ref{lem:interval} are fulfilled when $N\ge N_0(r)$. Then
$$
\max_{t > 0} f(t) = \max_{\alpha \in [1,2]}f(t_\alpha).
$$
Let $N\ge N_0(r)$,
$$
s = \log N, \quad g(s) = \frac{\log (rs)}{s+1}, \quad
\text{and}\quad g_0(s) = \frac{\log ((r+2)s)}{s+1}.
$$
Then $s>1$, $0<g(s)<g_0(s)$, $t_\alpha = t_0\left( 1+\alpha g(s) \right )$ and
\begin{align*}
f(t_\alpha) & = t_\alpha^{2N-t_\alpha}(2re)^{t_\alpha}
\\[1ex]
& = t_0^{2N-t_0(1+\alpha g(s))}(1+\alpha g(s))^{2N-t_0(1+\alpha g(s))}
(2re)^{t_0(1+\alpha g(s))}
\\[1ex]
&=
f(t_0)(e^{\alpha g(s)})^{-t_0\log t_0} 
(1+\alpha g(s))^{2N-t_0(1+\alpha g(s))}
(2re)^{t_0\alpha g(s)}.
\end{align*}
Since $1+x\leq e^x$ we obtain
\begin{align*}
f(t_\alpha) &\leq f(t_0)\left(e^{\alpha g(s)}\right)^{-t_0\log t_0 + 2N}
(1+\alpha g(s))^{- t_0 (1 + \alpha g(s))}(2re)^{t_0\alpha g(s)}
\\[1ex]
&\leq f(t_0)\left(e^{\alpha g(s)}\right)^{-t_0\log t_0 + 2N}
(2re)^{t_0\alpha g(s)}.
\end{align*}

\par

Now
\begin{equation*}
2N-t_0\log t_0
= \frac{2N}{\log N}\left( \log \left ( \frac{\log N}{2} \right ) \right) \leq
\frac{2N}{\log N}\log s,
\end{equation*}
which gives
$$
f(t_\alpha) \leq f(t_0)\left ( e^{\alpha g(s)\log s}\right ) ^{\frac{2N}
{\log N}}(2re)^{\frac{2N}{\log N}\alpha g(s)}\le 
f(t_0) \theta _0(r)^{\frac{2N}{\log N}},
$$

when
$$
\theta _0(r) = \sup _{s>1}e^{\alpha g_0(s)\log s} (2re)^{\alpha g_0(s)}.
$$
Evidently, $\theta _0$ is positive and increasing on $[0,\infty )$,
and since
$$
f(t_0) = \left ( \frac{2N}{\log N}\right ) ^{2N(1 - \frac{1}{\log N})}(2re)^{\frac{2N}{\log N}},
$$
\eqref{maxf(t)Est} holds with $\theta (r)= 2e\theta _0(r)$.
\end{proof}

\par

\begin{lemma}\label{lem:series}
Let $0 < r < r_0$, $a_1>0$, $a_2\ge 0$ and let $\theta (r)$
be as in Lemma
\ref{lem:maximum}.
Then
\begin{multline}\label{Eq:lem:series}
\sum_{k=0}^\infty \frac{(a_1k+a_2)^{2N}r^{2k}}{k!}
\\[1ex]
\leq 
C a_1^{2N}\left(\frac{2N}{\log N}\right)
^{2N(1 - \frac{1}{\log N})} (r_0^2 \cdot \theta (r_0^2))^\frac{2N}{\log N}
\frac {r_0^2}{r_0^2-r^2},
\end{multline}
for some constant $C\ge 1$, depending on $a_1$, $a_2$ and $r_0$ only,
provided $N$ is chosen large enough.
\end{lemma}

\par

\begin{proof}
First we prove the result in the case $a_1=1$ and
$a_2=a\in {\bf N}$. By Lemma \ref{lem:maximum} and using
Stirling's formula we have 
\begin{multline*}
\sum_{k=0}^\infty\frac{(k+a)^{2N}r^{2k}}{k!} = r^{-2a}
\sum_{k=0}^\infty\frac{(k+a)^{2N}r^{2(k+a)}}{k!}
\\[1ex]
\leq C_1 r^{-2a}\sum_{k=0}^\infty\frac{(k+a)^{2N}(2r^2)^{k+a}}{(k+a)!} 
\\[1ex]
= C_1 r^{-2a}\sum_{k=a}^\infty\frac{k^{2N}(2r^2)^k}{k!}
\leq C_2 r^{-2a}\sum_{k=a}^\infty\frac{k^{2N}(2r_0^2e)^k}{k^k}
\left ( \frac{r}{r_0}\right )^{2k}
\\[1ex]
\leq C_2 r^{-2a}\left( \frac{2N}{\log N} \right)^{2N (1 - \frac{1}{\log N} )}
(r_0^2\cdot \theta (r_0^2))^{\frac{2N}{\log N}} \sum_{k=0}^\infty
\left ( \frac{r}{r_0} \right )^{2k}
\\[1ex]
= C_2\left(\frac{2N}{\log N}\right)^{2N(1 - \frac{1}{\log N})} 
(r_0^2 \cdot \theta (r_0^2))^{\frac{2N}{\log N}}
\frac{r_0^2}{r_0^2-r^2},
\end{multline*}
for some constants $C_1$ and $C_2$ which are independent of
$N$ and $r_0$.

\par

For general $a_1$ and $a_2$, let
$$
s_{a_1,a_2}(N,r) := \sum _{k=0}^\infty \frac {(a_1k+a_2)^{2N}r^{2k}}{k!}
\quad \text{and}\quad s_a(N,r):= s_{1,a}(N,r).
$$
Then $s_{a_1,a_2}(N,r)=a_1^{2N}s_{a_2/a_1}(N,r)$, and hence
it suffices to prove \eqref{Eq:lem:series} in the case $a_1=1$.
Moreover, since $s_{a_1,a_2}(N,r)$ increases with $a_1$ and $a_2$,
and all factors on the right-hand side of \eqref{Eq:lem:series} except
$C$ are independent of $a_1$ and $a_2$, it follows that
we may assume that $a=a_2$ is an integer, and the proof is complete.
\end{proof}

\begin{rem}

For a given $N\in {\mathbb N},$ the Bell number $B_N$ counts the number of all partitions of a set of size $N$ and it is given by $$B_N=\frac{1}{e}\sum _{k=1}^\infty \frac{k^{2N}}{k!}.$$ 
According to \cite[Theorem 2.1]{berend-tassa},
$$
\frac{1}{e}\sum _{k=1}^\infty \frac{k^{2N}}{k!} <
\left ( \frac{0.792\times 2N}{\log(2N+1)} \right ) ^{2N},\ N = 1,2,\ldots
$$
For large values of $N$ we claim that the estimate of Lemma \ref{lem:series}
is significantly better. 

For arbitrary
$$
r > 0,\quad \frac{1}{e} < \lambda < 1
\quad \text{and}\quad
0 < a < 1-\log \left (\frac{1}{\lambda} \right )=1+\log \lambda 
$$
we have
$$
\left ( \frac {2N}{\log N} \right )^{2N} \left ( \frac {r\log N}{2N} \right )
^{\frac {2N}{\log N}}
= \left ( \frac {2\lambda N}{\log(2N+1)} \right )^{2N}\cdot C
_N^{\frac {2N}{\log N}},
$$
where
$$
\lim_{N\to \infty}C_N\cdot N^a = 0.
$$

\par

In fact,
$$
C_N = \frac{r}{2} \left ( \frac{\log(2N+1)}{\log N} \right ) ^{\log N}
\frac{\log N}{e^{\log N (1 + \log({\lambda}))}}.
$$
Since
$$
\lim_{N\to \infty}\left(\frac{\log(2N+1)}{\log N}\right)^{\log N} = 2
$$
and
$$
\lim_{N\to \infty}\frac{N^a \log N}{e^{\log N (1 + \log({\lambda}))}} =
\lim_{N\to \infty}\frac{\log N}{e^{\log N (1 + \log({\lambda})-a)}} = 0
$$
we are done.

\par

The lower estimate
$$
\left(\frac{2N}{e \log(2N)} \right)^{2N} \leq
\frac{1}{e}\sum _{k=1}^\infty \frac{k^{2N}}{k!}
$$
appears in  \cite[(2.4)]{berend-tassa}. Since the function $f(t)=t^{2N-t}$ is increasing in
$[1,1+\frac{2N}{\log N}]$, we have for large values of $N$
$$
f(\frac{2N}{\log N})\leq \frac{1}{e}\sum_{k=1}^\infty \frac{k^{2N}}{k!}.
$$
Again it is straight-forward to check that
$$
\left(\frac{2N}{e \log(2N)} \right)^{2N}=o(f(\frac{2N}{\log N}))\quad \text{as}\quad N\to \infty.
$$
\end{rem}

\par

\par\medskip

In the following we apply the previous result to functions $f$ with Hermite
series expansions, given by \eqref{fHermite}.

\par

\begin{prop}\label{prop1}
Let $f\in L^2(\rr d)$ be given by \eqref{fHermite} such that
$$
\left | c_\alpha (f)\right| \lesssim r^{|\alpha |} (\alpha !)^{-\frac{1}{2}}
$$
for some $r > 0.$ Then, for $0<d\cdot r<r_0,$
\begin{gather*}
\norm{H^N f}^2_{L^2} \lesssim 2^{2N}\left(\frac{2N}{\log N}\right)
^{2N(1 - \frac{1}{\log N})} (r_0^2 \cdot \theta (r_0^2))^\frac{2N}{\log N}
\frac {r_0^2}{r_0^2-(d\cdot r)^2}.
\end{gather*}
\end{prop}

\par

\begin{proof}
From
$$
|c_\alpha (H^N f)| = (2|\alpha |+d)^N|c_\alpha (f)|
$$
we obtain
\begin{align*}
\norm{H^N f} _{L^2}^2 
&= \sum _{\alpha \in \nn d}
|c_\alpha (H^N f)|^2 \lesssim \sum_{\alpha \in \nn d} \frac{(2|\alpha |+d)
^{2N}r^{2|\alpha |}}{\alpha !}
\\[1ex]
&\lesssim \sum _{\alpha \in \nn d} \frac{(2|\alpha |+d)^{2N}
d^{|\alpha |}r^{2|\alpha |}}{|\alpha | !}
\asymp \sum _{k=0}^\infty \frac{(2k+d)^{2N}(d\cdot r)^{2k}}{k!},
\end{align*}
and it suffices to apply Lemma \ref{lem:series}.
\end{proof}

\par

Next we deduce some kind of converse of Proposition
\ref{prop1}. For this reason we need the following lemma.

\par

\begin{lemma}\label{lem:convex}
Let $r > 1$,
$$
\psi (t) = e^t \left( \frac {t}{2}-\log r \right ) + r^2\log r
\quad \text{and}\quad
t_0 = 2\log r.
$$
Then there exists a convex and increasing
function $\varphi:[0,\infty)\to [0,\infty)$ such that $\varphi(0) = 0$
and $\varphi(t) = \psi(t)$ for every $t\geq t_0$.
\end{lemma}

\par

\begin{proof}
We have
$$
\psi ^\prime(t) = e^t \left ( \frac{1+t}{2} - \log r\right )
\quad \mbox{and}\quad
\psi ''(t) = e^t \left (1 + \frac{t}{2} - \log r \right ).
$$
Hence
$$
\psi ''(t) > \psi ^\prime(t) > 0
$$
for every $t\geq t_0$. On the other hand, the tangent line
to the graph of $\psi$ at point $(t_0, \psi(t_0))$ passes through
$(0,0),$ since $\psi(t_0) = t_0 \psi^\prime(t_0)$. Consequently,
if $\varphi : [0,\infty)\to [0,\infty)$ is defined by $\varphi(0) = 0$,
$\varphi$ linear in $[0, t_0]$ and $\varphi (t) = \psi(t)$ for $t\geq t_0$,
then $\varphi$ satisfies the required conditions.
\end{proof}

\par

\begin{prop}\label{HermiteEstToCoeffEst}
Let
$$
\norm{H^N f}_{L^2} \lesssim 2^N
\left (\frac{2N}{\log N} \right )^{N(1-\frac{1}{\log N} ) }r^{\frac{2N}{\log N}}.
$$
Then
$$
\left |c_\alpha (f)\right | \lesssim r^{|\alpha |} |\alpha |^{-\frac{|\alpha |}{2} + 1}.
$$
\end{prop}

\par

Before the proof we recall that the Young conjugate of a convex and increasing function
$\varphi : [0,\infty )\to [0,\infty )$ is the increasing and convex function
$\varphi^\ast : [0,\infty )\to [0,\infty )$, given by
$$
\varphi ^\ast (s) =
\sup _{t\geq 0}\left ( st - \varphi (t) \right ) .
$$
It turns out that $( \varphi ^\ast )^\ast = \varphi$.

\par

Now assume that $\varphi$ is the same as in Lemma \ref{lem:convex}.
We claim that
\begin{equation}\label{phi*Est}
k^{\frac{k}{2}-1}r^{-k} \lesssim \exp \left ( \sup_N \left (N \log k -
\varphi ^\ast (N) \right ) \right ) ,
\end{equation}
when $k$ is large enough.

\par

In fact, for $k$ large enough we have
$$
\exp \left( \varphi (\log k)\right ) = k^{\frac{k}{2}}r^{-k}r^{r^2}
\asymp k^{\frac{k}{2}}r^{-k}.
$$
Moreover, for $s\in [N, N+1]$ we have
$$
st - \varphi ^\ast(s) \leq t + \left (Nt - \varphi^\ast(N)\right ),
$$
from where it follows
$$
\varphi (t) = \sup _{s\geq 0} \left ( st - \varphi^\ast(s) \right )
\leq t + \sup _N \left ( Nt - \varphi^\ast(N) \right ).
$$
In particular,
$$
k^{\frac{k}{2}}r^{-k} \lesssim \exp \left (\log k + \sup_N \left (N\log k -
\varphi^\ast(N)\right ) \right ) ,
$$
which is the same as \eqref{phi*Est}.

\par

\begin{proof}[Proof of Proposition \ref{HermiteEstToCoeffEst}]
Let
\begin{equation}\label{EqEpN}
\varepsilon _N := \left ( \frac {\log N}{2N} \right )^{\frac{N}{\log N}},
\end{equation}
and let $\fy$ be the same as in Lemma \ref{lem:convex}.
Since
$$
\left | c_\alpha (f)\right | \leq \inf_N
\left (\frac{1}{(2|\alpha |)^N}\norm {H^N{\color{red}f}}_{L^2} \right )
$$
we have
\begin{equation*}
\frac 1{|c_\alpha (f)|} \gtrsim \exp \left (\sup _N \left ( N\log |\alpha | - \frac{2N}{\log N}\log r - N
\log \left ( \frac{2N}{\log N} \right ) - \log \varepsilon _N \right ) \right ) .
\end{equation*}
%
%
%
%
By Lemma \ref{lem:convex} and \eqref{phi*Est} it suffices to prove
$$
\frac{2N}{\log N}\log r + N\log \left (\frac{2N}{\log N} \right ) + \log \varepsilon_N
\le \varphi ^*(N) + C,
$$
for large values of $N$, where $C>0$ is a constant which is independent of
$N$ and $k$. From the definitions it follows that, for $N$ large enough,
\begin{align*}
\varphi ^\ast(N) & \geq N\log \left ( \frac{2N}{\log N} \right ) - \varphi
\left ( \log \left ( \frac{2N}{\log N} \right ) \right )
\\[1ex]
&= N \log \left ( \frac{2N}{\log N} \right ) + \log ( \varepsilon _N)
+ \frac{2N}{\log N}\log r - r^2\log r
\end{align*}
and the lemma is proved.
\end{proof}

\par

\begin{proof}[Proofs of Theorems \ref{thm:mainthm1} and \ref{thm:mainthm2}]
The results follow immediately from Propositions \ref{prop1} and
\ref{HermiteEstToCoeffEst}.
\end{proof}

\section{Mapping properties of $\bsySig _s(\rr d)$ and
$\bsycalS _s(\rr d)$ under the Bargmann transform}\label{sec3}

\par

In \cite{Toft14}, complete mapping properties of $\bsySig _s(\rr d)$
and $\bsycalS _{\! s}(\rr d)$ under the Bargmann transform are
deduced when $s\ge 1/2$ and when $s=0$. Here we show
analogous properties in the case $0<s<1/2$.

\par

In what follows we let
\begin{align}
\maclA _s(\cc d) &:=\sets {F\in A(\cc d)}{|F(z)|\lesssim e^{R(\log
\eabs z)^{\frac 1{1-2s}}}
\ \text{for some}\ R>0}\label{IndLimAnalSpace}
\intertext{endowed with the inductive limit topology, and}
\maclA _{0,s}(\cc d) &:=
\sets {F\in A(\cc d)}{|F(z)|\lesssim e^{R(\log \eabs z)^{\frac 1{1-2s}}}
\ \text{for every}\ R>0}\label{ProjLimAnalSpace}
\end{align}
endowed with the projective limit topology, when $0<s<\frac 12$. Here,
$\langle z \rangle :=(1+|z|^2)^{1/2}$, as usual.

\par

We may replace the inequalities in \eqref{IndLimAnalSpace}
and \eqref{ProjLimAnalSpace} by suitable (weighted) $L^p$ estimates
on the involved entire functions. This is for example a consequence of
\cite[Theorem 3.2]{Toft12} and the fact that
$$
\sets {e^{\frac {|z|^2}2 -R(\log \eabs z)^{\frac 1{1-2s}}}}{R>0}
$$
is an admissible family of weight functions on $\cc d$ in the sense of
\cite[Definition 1.4]{Toft12}.

\par

Therefore, for any $p\in [1,\infty]$, 
$\maclA _s(\cc d)$ is the set of all $F\in A(\cc d)$ such that
\begin{equation}\label{AnalLpEst}
\left (\int _{\cc d}  |F(z)e^{-R(\log \eabs z)^{\frac 1{1-2s}}}|^p\, d\lambda (z)
\right )^{1/p}<\infty
\end{equation}
is true for some $R>0$, and $\maclA _{0,s}(\cc d)$ is the set of all
$F\in A(\cc d)$ such that \eqref{AnalLpEst} is true for every $R>0$
(also in topological sense, and with obvious modifications when
$p=\infty$). We also note that $\maclA _{0,s}(\cc d)$ is a Fr{\'e}chet
space.

\par

\begin{thm}
Let $0<s<\frac 12$. Then the following is true:
\begin{enumerate}
\item The Bargmann transform $\mathfrak V_d$ is a topological isomorphism from
$\bsycalS _{\! s}(\rr d)$ onto $\maclA _{0,s}(\cc d)$;

\vrum

\item The Bargmann transform $\mathfrak V_d$ is  a topological isomorphism from
$\bsySig _{s}(\rr d)$ onto $\maclA _s(\cc d)$.
\end{enumerate}
\end{thm}

\par

%
%

\begin{proof}
Let
$$
\vartheta_R(\alpha):= \left ( \frac {\pi ^d}{2^d(|\alpha |+d-1)!}
\int _0^\infty e^{-R(\log \eabs r)^{\frac 1{1-2s}}} r^{|\alpha |+d-1}\, dr \right )^{1/2}
$$
and denote by $\maclH _{R}^2(\rr d)$ the set of all $f\in \mascS (\rr d)$ such that
$c_\alpha (f)$ in \eqref{fHermite} satisfies
$$
\nm f{\maclH _{R}^2}:= \left (  \sum _{\alpha \in \nn d}
|c_\alpha (f)\vartheta_R (\alpha )|^2\right )^{1/2}  <\infty .
$$
%
%
Then it follows from Proposition 3.4 in \cite{Toft14} that  $\mathfrak V_d$ is a
topological isomorphism from
$$
\bigcup _{R>0}\maclH _{R}^2(\rr d)
$$
onto $\maclA _s(\cc d)$ 
and from 
$$
\bigcap _{R>0} \maclH _{R}^2(\rr d)
$$
onto 
$\maclA _{0,s}(\cc d)$. Hence, to conclude, it suffices to show
the existence of positive constants $c_j$ and $a_j$, $j=1,2$ such that
\begin{equation}\label{varthetaEst}
e^{c_1|\alpha |^{\frac 1{2s}}/R^{a_1}} \lesssim \vartheta_R (\alpha )
\lesssim e^{c_2|\alpha |^{\frac 1{2s}}/R^{a_2}},
\end{equation}
for every $R>0$.

\par

In order to prove \eqref{varthetaEst}, let
$$
m_\alpha=|\alpha |+d-1\quad \text{and}\quad \theta =\frac 1{1-2s}>1.
$$
and fix $1<\mu <\theta$. Then
$$
\vartheta_R (\alpha )^2 \gtrsim \frac{1}{m_\alpha !}\int
_{R{_1,\alpha}}^{R_{2,\alpha}}e^{-R(\log  r)^{\theta}}
r^{m_\alpha}\, dr,
$$
where
$$
\log R_{2,\alpha}=\left ( \frac{m_\alpha}{\theta R} \right )^{\frac{1-2s}{2s}}
\quad \text{and}\quad
\log R_{1,\alpha}=\frac{1}{\mu}\log R_{2,\alpha}.
$$
For $R_{1,\alpha} \leq r \leq R_{2,\alpha}$ we have
\begin{gather*}
-R(\log r)^\theta+m_\alpha \log r \geq m_\alpha
\log R_{1,\alpha}-R (\log R_{2,\alpha})^\theta
\\[1ex]
=
\frac{m_\alpha}{\mu}\left ( \frac{m_\alpha}{\theta R} \right )
^{\frac{1-2s}{2s}}
-R \left (\frac{m_\alpha}{\theta R} \right )^{\frac{1}{2s}}
=
\theta ^{-\frac{1}{2s}}m_\alpha^{\frac{1}{2s}}R^{1-\frac{1}{2s}}
\left ( \frac{\theta}{\mu}-1\right ) \geq \frac{C_1|\alpha |
^{\frac{1}{2s}}}{R^{a_1}},
\end{gather*}
where $a_1=\frac{1}{2s}-1$ and $C_1=\theta
^{-\frac{1}{2s}}(\frac{\theta}{\mu}-1)$.

\par

As $R_{2,\alpha}-R_{1,\alpha}
=R_{2,\alpha}(1-R_{2,\alpha}^{\frac{1}{\mu}-1})$  increases with
$|\alpha|$ and $\log (m_\alpha !)=o(|\alpha|^{\frac{1}{2s}})$ as
$|\alpha|\to \infty,$ we get (with $c_1<C_1$ fixed)  that
$$
\vartheta_R(\alpha)  \geq C_R e^{c_1|\alpha |^{\frac 1{2s}}/R^{a_1}}.
$$ 
  
\par
  
To prove the other inequality we observe that for each $R>0$, there
is a constant $C>0$ such that
\begin{align*}
\vartheta_R^2(\alpha)
&\leq C \int_e^\infty
e^{-\frac{R}{2}(\log r)^\theta}g_\alpha(r)\, dr
\\[1ex]
&\leq C \sup _{r \ge e} \big ( g_\alpha(r)\big )\left ( \int_e^\infty e^{-\frac{R}{2}(\log r)^\theta}\, dr
\right ) ,
\end{align*}
where $g_\alpha(r)= e^{-\frac{R}{2}(\log r)^\theta}r^{m_\alpha}$.

\par
  
By straight-forward computations it follows that $g_\alpha (r)$
when $r\ge e$, attains its global maximum for
$$
r_\alpha=\exp \left ( \left (\frac {2m_\alpha}{\theta R}  \right )
^{\frac {1-2s}{2s}} \right ),
$$
and that
$$
g_\alpha(r_\alpha) = e^{c_2m_\alpha^{\frac 1{2s}}/R^{a_2}},
$$
where
$$
c_2=2^{\frac 1{2s} -1}\theta ^{-\frac 1{2s}} (\theta -1)
\quad \text{and}\quad
a_2= \frac 1{2s}-1.
$$
Hence, by replacing $c_2$ by a larger constant, if necessary,
we get the desired inequality.
\end{proof}

\par

\end{document}